\documentclass{amsart}

\usepackage{graphicx}

\newtheorem{theorem}{Theorem}[section]
\newtheorem{lemma}[theorem]{Lemma}

\theoremstyle{definition}
\newtheorem{definition}[theorem]{Definition}

\theoremstyle{remark}

\numberwithin{equation}{section}

\usepackage{bm}
\usepackage{bbm}
\newcommand{\N}{\mathbb{N}}
\newcommand{\limn}{\lim_{n\to\infty}}

\newcommand{\indi}{\mathbbm{1}}
\newcommand{\C}[1]{\mathcal{#1}}

\usepackage{tikz}
	\usetikzlibrary{positioning}
    \usetikzlibrary{calc}
\usepackage{tikz-qtree}
\usepackage{subfigure}

\begin{document}

\title{Introducing Minkowski normality}

\author{Dajani, K.}
\address[Dajani, K.]{Department of Mathematics, P.O. Box 80.010, 3508 TA Utrecht,
The Netherlands}
\curraddr{Department of Mathematics, Mathematical Institute, Utrecht University, 3584 CD Utrecht, the Netherlands}
\email{k.dajani1@uu.nl}
\thanks{}

\author{de Lepper, M.R.}
\address[de Lepper, M.R.]{Department of Mathematics, Mathematical Institute, Utrecht University, 3584 CD Utrecht, the Netherlands}
\email{mathijsdelepper@gmail.com}
\thanks{The A.F. Monna Fund, De Fundatie van de Vrijvrouwe van Renswoude te Delft, The George Washington University.}

\author{Robinson, Jr., E. A. }
\address[Robinson, jr., E.A.]{Department of Mathematics, 
George Washington University,
 Washington D.C. 20052, United States of America}
\email{robinson@email.gwu.edu}
\thanks{The Simons Foundation, award no. 244739}


\date{January 4, 2019 and, in revised form, Month Day, Year.}

\dedicatory{This paper is dedicated to our advisors.}


\begin{abstract}
We introduce the concept of Minkowski normality, a different type of normality for the regular continued fraction expansion. We use the ordering 
\[ \frac{1}{2},\quad \frac{1}{3}, \frac{2}{3},\quad \frac{1}{4}, \frac{3}{4},\frac{2}{5}, \frac{3}{5},\quad \frac{1}{5}, \cdots \]
of rationals obtained from the Kepler tree to give a concrete construction of an infinite continued fraction whose digits are distributed according to the Minkowski question mark measure. To do this we define an explicit correspondence between continued fraction expansions and binary codes to show that we can use the dyadic Champernowne number to prove normality of the constructed number. Furthermore, we provide a generalised construction based on the underlying structure of the Kepler tree, which shows that any construction that concatenates the continued fraction expansions of all rationals, ordered so that the sum of the digits of the continued fraction expansion are non-decreasing, results in a number that is Minkowski normal.
\end{abstract}

\maketitle

\section{Introduction}
Over the years, many constructions have been done both of normal numbers, as introduced by 
Borel, as for other types of normality. These different types of normality correspond to different 
number expansions and different measures. Over the years there have been many explicit constructions of normal numbers, both in the sense of  Borel, as well as for other types of normality. These different types of normality correspond to different number expansions and 
to different measures. The concrete constructions that have been developed are all associated to a distribution that result from Lebesgue measure or, in the case of regular continued fractions, the absolutely continuous Gauss measure. However, in this article, we consider a measure that is singular with respect to Lebesgue measure. We consider the \emph{Minkowski question mark measure} 
$\mu_?$, which is specified by the following distribution function
\[ ?(x) := 2\displaystyle\sum_{i=1}^{\infty} \frac{(-1)^{i+1}}{2^{a_1(x)+a_2(x)+\cdots+a_i(x)}}. \]
Here, $a_i(x)$ comes from the \emph{continued fraction expansion} of $x \in [0,1)$, $i\geq 1$. In particular, we introduce a different type of normality for regular 
continued fraction expansions that we call \emph{Minkowski normality}. Informally, we say that a number $x$ is Minkowski normal if its  digits $(a_i(x))_{i\geq1}$ are distributed according to the Minkowski question mark measure. \\ 

The main goal of the article is to construct explicitly a Minkowski normal number. We construct an infinite continued fraction expansion and show that the corresponding sequence of digits is distributed according to the Minkowski question mark measure. Specifically, we consider the ordering of rationals that is given by the \emph{Kepler tree}. This is a specific binary tree that orders the rationals in the unit interval. The constructed number is obtained by concatenating the continued fraction expansions of the rationals using the Kepler order. For the proof of normality, we show that there is a correspondence between binary codes and rationals in the Kepler tree. Moreover, we show that we can use the dyadic Champernowne number to determine the distribution of the sequence of digits that represent the constructed number. 
Finally, we use generalised Champernowne numbers to extend normality of the constructed number to more general cases. \\

\emph{The work in this article is based on the master's thesis of M.R. de Lepper, which was completed in partial fulfilment of his degree in Mathematical Sciences at Utrecht University. His gratitude goes out to K. Dajani, who supervised him and has been at the foundation of the work.}

\section{Normality and continued fractions}
Normality as introduced by Borel, focusses on integer base expansions. We say that $x \in [0,1)$ is \emph{normal in base $b$} if for any block $d=d_1d_2\cdots d_k$ of $k$ digits, $d_i \in \{0,1,\ldots, b-1 \}$, we have
\[ 
\limn \frac{1}{n} G_n(x,d) = b^{-k}. 
\]
Here, $G_n(x,d)$ denotes the number of occurrences of $d$ in the first $n$ digits of the base $b$ expansion of $x$. Over the years, many explicit constructions of normal numbers have been made. The first and most well-known is due to Champernowne \cite{Champernowne1933}. He proved that the number that is obtained by concatenating the natural numbers, i.e.
\begin{equation*}
\C C_{10} = 0.
\text{ } 1 \text{ } 2\text{ } 3\text{ } 4 \text{ }5\text{ } 6 \text{ } 7 \text{ } 8 \text{ } 9 \text{ } 10 \text{ } 11 \text{ } 12 \text{ } 13 \text{ } 14 \text{ }  \cdots,
\label{eqn:champernowne}
\end{equation*}
is normal in base 10. 
Later, Copeland and Erd\"os gave a generalised construction of a normal number 
\cite{CopelandErdos1946}, which they used to prove the normality of the number that is obtained by 
concatenating all the primes. A small selection of further generalisations and results include that of Davenport and Erdos \cite{Davenport1952} and Nakai and Shiokawa \cite{Nakai1992}. \\

The definition of normality can be extended to continued fractions. Any real number $x$ can be represented as a - possibly finite - continued fraction expansion
\[ x = \cfrac{1}{ a_1(x) + 
			\cfrac{1}{ a_2(x) + 
				\cfrac{1} {a_3(x) + \cfrac{1}{\ddots}} 
				}
				}
,
\]
where the digits $a_i(x) \in \mathbb{N}$ are the \emph{partial quotients} of $x$,  $i \geq1$. In shorthand, we write $x= [a_1, a_2, a_3, \cdots]$. For any irrational $x$, the continued fraction expansion is infinite and unique \cite[Theorem 5.11]{Niven2014}. Moreover, any rational has exactly two expressions as a finite continued fraction $[a_1, a_2, \cdots, a_n-1, 1] = [a_1, a_2, \cdots, a_n]$. We use the convention that any rational continued fraction is written in its \emph{reduced form}: the one on the right, where $a_n \geq 2$. \\

The type of normality that is related to the continued fraction expansion comes from the \emph{Gauss measure} $\gamma$ that, for any Lebesgue set $A \subset [0,1)$, is defined by
\begin{equation}
\gamma(A) := \frac{1}{\log2} \int_A \frac{1}{1+x} dx.
\label{eq:Gaussmeasure}
\end{equation}
Therefore, we say that $x \in [0,1)$ is \emph{continued fraction normal}, if for any 
$k\ge 1$ and any block $d = d_1, d_2,\cdots, d_k$, $d_i \in \N$, we have
\[ 
\limn \frac{1}{n} G_n(x,d) = \gamma(\Delta(d)), 
\]
where $\Delta(d) = \{ y \in [0,1) : y = [d_1, d_2, \cdots, d_k, \cdots] \}$ is the cylinder set 
corresponding to $d$. In the above and henceforth, $G_n(x,d)$ will denote the number of occurrences of $d$ in the first $n$ digits of the continued fraction expansion of $x$. It follows from Birkhoff's Ergodic Theorem that Lebesgue almost all numbers are continued fraction normal. \\  

In contrast to the case of normality for radix base expansions, where there are 
a large number of explicit constructions of normal numbers, there are very few results to date 
about continued fraction normality. So far, there are four construction results. 
The first is due to Postnikov \cite{Postnikov1960}, who used Markov chains to construct a continued fraction normal number. Another, more explicit, construction is 
due to Adler, Keane and Smorodinsky \cite{Adler1981}. They first construct a (sub)sequence of rationals by taking all non-reduced fractions with denominator $n$ in increasing order
\begin{equation}
\frac{1}{2},\quad \frac{1}{3}, \frac{2}{3},\quad \frac{1}{4}, \frac{2}{4}, \frac{3}{4},\quad \frac{1}{5}, \frac{2}{5}, \frac{3}{5},\frac{4}{5}, \quad \ldots, \frac{n-1}{n},\quad \frac{1}{n+1}, \cdots.
\label{eqn:seqaks}
\end{equation}
Their continued fraction normal number is then obtained by concatenating the - finite - continued fraction expansions of these rationals
\[ x_{aks} = [2,\quad 3, 1, 2,\quad 4, 2, 1, 3,\quad 5, 2, 2, 1, 1, 2, 1, 4, \cdots] \approx 0.44034. \]

It took about 30 years before the constructions of Postnikov and Adler, Keane and Smorodinsky were generalized. The generalisation of Postnikov's construction, due to Madritsch and Mance \cite{Madritsch2016}, introduces a generalised form of normality. Neither of these works construct an explicit number that is continued fraction normal. This is different from the work of Adler, Keane and Smorodinsky and the generalisation of their work, which is due to Joseph Vandehey \cite{Vandehey2016a}. Among other things, Vandehey proves that some explicit subsequences of \eqref{eqn:seqaks} can be used to construct a continued fraction normal number. For the proof, Vandehey uses metrical results to get asymptotics on how many rationals have good small-scale properties. In turn, these asymptotics imply conditions that determine whether the constructed number is continued fraction normal. One of the constructions for instance, considers the subsequence of rationals that have integer numerators and prime denominators. Though Madritsch and Mance provide a generalised form of normality that is applicable to continued fraction expansions, the constructions from Vandehey and Adler, Keane and Smorodinsky are the only known concrete constructions of continued fraction normal numbers. The next section, however, introduces a new type of normality for continued fractions and provides concrete constructions.

\section{The construction} \label{section:construction}
The crucial factor in determining the limiting distribution of the partial quotients of the constructed number, is the ordering that is chosen. In the case of Adler, Keane and Smorodinsky, the ordering of rationals they use leads to normality with respect to the Gauss measure. Hence, the constructed number is continued fraction normal. In this section, we consider the ordering of the rationals that results from the \emph{Kepler tree}. We use this ordering to construct a number whose partial quotients are distributed according to the Minkowski question mark measure. \\

The first part of the Kepler tree is found in Johannes Kepler's magnum opus, a book containing his most important work. See \cite[p. 163]{Kepler1997} for an English translation. Though Johannes Kepler starts from $1/1$, the binary tree starts from $1/2$ and then uses the rule
\begin{center}
\begin{tikzpicture}[scale = 0.55]
  \node {$p/q$}[sibling distance=70mm]
    child {node {$p/(p+q)$}}
    child {node {$q/(p+q)$}};
\end{tikzpicture}
.
\end{center}
As rationals can be represented by finite continued fractions and vice versa, this is equivalent to
\begin{center}
\begin{tikzpicture}[scale = 0.55]
  \node {$[a_1,a_2,\cdots, a_n]$}[sibling distance=70mm]
    child {node {$[(a_1+1),a_2,\cdots, a_n]$}}
    child {node {$[1,a_1,a_2,\cdots, a_n]$}};
\end{tikzpicture}
.
\end{center}
%
This representation allows us to understand the behaviour of the sequence of digits that is obtained from the construction. Here, note that a left move increases the first digit in the continued fraction by one and does not alter the total number of digits in the continued fraction. A right move however, inserts a 1 as a first digit and thus increases the length of the continued fraction by one. This also means that a left move does not preserve the block of digits that form the continued fraction of the mother node, whereas a right move does preserve the block. Lastly, note that both moves increase the sum of the digits of the continued fraction expansion by one. Hence, the Kepler tree orders the rationals into levels based on the sum of the digits of their continued fraction expansion. The first four levels of the tree are displayed in Figure \ref{fig:keplertree}. \\
\begin{figure}[!htb]
\begin{center}
{
\begin{tikzpicture}[level distance=20mm, scale = 0.8]
  \tikzstyle{every node}=[fill=white!60,circle,inner sep=1pt]
  \tikzstyle{level 1}=[sibling distance=60mm,
    set style={{every node}+=[fill=white!45]}]
  \tikzstyle{level 2}=[sibling distance=30mm,
    set style={{every node}+=[fill=white!30]}]
  \tikzstyle{level 3}=[sibling distance=15mm,
    set style={{every node}+=[fill=white!15]}]
  \node {1/2}
     child {node {1/3}
       child {node {1/4}
         child {node {1/5}}
         child {node {4/5}}
       }
       child {node {3/4}
         child {node {3/7}}
         child {node {4/7}}
       }
     }
    child {node {2/3}
       child {node {2/5}
         child {node {2/7}}
         child {node {5/7}}
       }
       child {node {3/5}
         child {node {3/8}}
         child {node {5/8}}
       }};
 
\end{tikzpicture}
}
%
 .
 \caption{The first 4 levels of the Kepler tree.}
 \label{fig:keplertree}
  \end{center}
\end{figure}
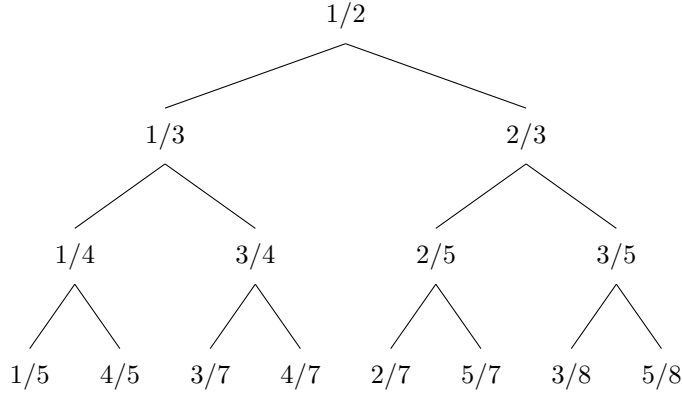

The key idea in proving normality of the constructed number is that we create a one-to-one correspondence between rationals and binary codes. This correspondence is based on the fact that there exists a unique path between the root and any rational in the Kepler tree. In turn, we use this unique path to define a one-to-one correspondence between rationals and binary codes. \\ 

The root corresponds to the empty path and therefore to the empty binary code. Moreover, given an arbitrary rational, we can retrace its path as follows. Let $[a_1, a_2, \cdots, a_n]$ denote the continued fraction of an arbitrary rational $p/q$ in the Kepler tree. Then by going $(a_1 - 1)$ steps from the left up, we end up at the rational that corresponds to $[1, a_2, \cdots, a_n]$. Subsequently, going from the right up we end at $[a_2, a_3, \cdots, a_n]$. By repeating this proces for $a_2, a_3, \ldots, a_{n-1}$ and $a_n$ we can find the path to the root. We summarise these steps symbolically by writing $L$ for a left move and $R$ for a right move. Subsequently, we reverse the path and apply the substitution $\{ L \mapsto 0, R \mapsto 1 \}$ to associate a binary code to $p/q$. Hence,
%
%
\begin{align*}
p/q \overset{\text{cfe}}{\longleftrightarrow} [a_1, a_2, \cdots, a_n]  &\overset{\text{upward path}}{\longleftrightarrow} L^{a_1-1}R L^{a_2 -1}R \cdots L^{a_n-2} \\
& \overset{\text{downward path}}{\longleftrightarrow} L^{a_n-2} \cdots R L^{a_2-1} R L^{a_1 -1} \\
 &\overset{\text{binary code}}{\longleftrightarrow}  0^{a_n-2} \cdots 1 0^{a_2-1} 1 0^{a_1 -1}.
\end{align*}
%
%
%
%
%

The binary code that is associated to a rational contains a lot of information. It gives the continued fraction expansion of the rational that it represents and its exact location within the tree. Namely, it gives the level in which the rational occurs and the position within that level. The level is given by the total number of $0$'s and $1$'s in its binary code and its position within the level can be read from the ordering of the  $0$'s and $1$'s. The following lemma is an immediate consequence of the binary coding and the concept of retracing paths in the tree.
\begin{lemma}
There exists a unique path between the root of the Kepler tree that starts at $1/2$ and any arbitrary rational $p/q$. If we denote $p/q$ by its continued fraction expansion $[a_1, a_2, \cdots, a_n]$, then the corresponding path is
\begin{equation}
L^{a_n-2} \cdots R L^{a_2-1} R L^{a_1 -1},
\label{eqn:keplerpath}
\end{equation}
which corresponds to the binary code
\begin{equation}
0^{a_n-2} \cdots 1 0^{a_2-1} 1 0^{a_1 -1}.
\label{eqn:d_end}
\end{equation}
This path consists of $a_1 +a_2+ \cdots + a_n -2$ moves, which also corresponds to the level in which the rational occurs for the first and only time.
\label{lemma:rationaloccurrence}
\end{lemma}
 
 Apart from providing information about the occurrence of rationals, the concept of retracing paths also tells us how blocks of the form $d=d_1,d_2, \cdots, d_k$ are formed by the Kepler tree, how these blocks are preserved and how we can identify them using binary codes. \\

For the construction, we order the rationals in the Kepler tree going top-down, left-right. The ordering of the rationals that result from this procedure is
\begin{equation}
\frac{1}{2},\quad \frac{1}{3}, \frac{2}{3},\quad \frac{1}{4}, \frac{3}{4},\frac{2}{5}, \frac{3}{5},\quad \frac{1}{5}, \cdots.
\label{eqn:seqrationals}
\end{equation}
If we concatenate the corresponding binary codes of these rationals in the given order, we obtain an infinite sequence of binary digits. This infinite sequence corresponds to the dyadic Champernowne number
\begin{equation}
\C C_2 := 0.
\text{ } 0 \text{ } 1\text{ } 00\text{ } 01 \text{ }10\text{ } 11 \text{ } 000  \text{ }  \cdots,
\label{eqn:binarychampernowne}
\end{equation}
which is known to be normal in base $2$. This and other properties of $\C C_2$ can for instance be found in \cite{Denker1992} or \cite{Shiokawa1975}. For our construction of a Minkowski normal number, we concatenate the continued fraction expansions of the rationals in the ordering that results from the Kepler tree. We obtain an infinite continued fraction, which corresponds to a unique irrational number \cite[Proposition 4.1.1]{Dajani2018}. This number is given by
\begin{equation}
\C K := [2,\quad 3,\quad 1, 2,\quad 4,\quad 1, 3,\quad  2, 2, \quad 1, 1, 2,\quad 5,\quad \cdots] \approx 0.44031.
\label{x}
\end{equation}

\section{Minkowski normality}
So far, different types of normality correspond to different number expansions. Next, however, we use the Minkowski question mark measure to define another type of normality for the continued fraction expansion. We define Minkowski normality for continued fractions as follows.
\begin{definition}[Minkowski normal number]
We say that $x = [a_1, a_2,a_3, \cdots] \in [0,1)$ is \emph{Minkowski normal}, if for any $k\ge 1$ and any block $d = d_1,d_2,\cdots ,d_k$, with $d_i \in \N$, we have that
\begin{equation}
\limn G_n(x,d) = \mu_?(\Delta(d)) = 2^{-(d_1+d_2 + \cdots + d_k)}.
\label{eqn:?normality}
\end{equation}
\label{def:?normality}
\end{definition}
\begin{theorem}
$\mu_?$ almost every number in $[0,1)$ is Minkowski normal.
\label{thm:Minkowski}
\end{theorem}
\begin{proof}
The Gauss map $\C T$ is known to be ergodic under the Minkowski question mark measure $\mu_?$. This follows from the fact that the Minkowski acts on cylinders as a product measure, which implies that we have an isomorphism with a Bernoulli shift. This gives Bernoullicity and hence mixing and ergodicity. Thus, let $x \in [0,1)$, $\C B$ the Borel $\sigma$-algebra on $[0,1)$ and consider the ergodic system $([0,1), \C B, \mu_?, \C T)$. Then for any $k\geq 1$ and any block $d=d_1, d_2, \cdots d_k$, $d_i\in \N$, it follows from Birkhoff's Ergodic Theorem that
\[ \limn\frac{1}{n}\sum_{i=0}^{n-1} \indi_{\Delta(d)} (\C T^i x) = \mu_?(\Delta(d)) = 2^{-(d_1+d_2 + \cdots + d_k)} \qquad{ \mu
_?\textrm{-}a.e.}
\]
\end{proof}

We note that $\lambda$-almost every number is continued fraction normal and $\mu_?$-almost every number is Minkowski normal. This is possible because Lebesgue measure and Minkowski question mark measure are singular. \\

The rest of the article is dedicated to proving the Minkowski normality $\C K$. To do this, we identify explicit binary codes that correspond to different types of occurrences of an arbitrary block $d$. Consequently, we use the base $2$ normality of $\C C_2$ to determine the frequency that corresponds to these type of occurrences. \\

We then distinguish the following four types of occurrences of a block $d$ in $\C K$: 
\begin{itemize}
\item The block $d$ occurs at the start of a continued fraction expansion of a rational in $\C K$;
\item The block $d$ occurs in the middle of the continued fraction expansion of a rational in $\C K$;
\item The block $d$ occurs at the end of the continued fraction expansion of a rational in $\C K$;
\item The block $d$ occurs in $\C K$ as a result of concatenating the continued fraction expansions of different rationals. We refer to this type of occurrences as \emph{divided occurrences}.
\end{itemize}

\begin{lemma}
Let $d = d_1,d_2,\cdots, d_k$  be an arbitrary block of length $k$, $d_i \in \N$. The asymptotic frequency of divided occurrences of $d$ in $\C K$ is equal to 0.
\label{lemma:concatfreq}
\end{lemma}
\begin{proof}
The $l$-th level of the Kepler tree consists of $2^l$ rationals. Hence, there are $2^l -1$ concatenations. As $d$ consists of $k$ digits, there is a maximum of $k-1$ positions where $d$ can be divided. Therefore, the number of divided occurrences can be bounded from above by $k 2^l$. \\

Each rational in the $l$-th level of the tree is formed by $i$ left moves and $l-i$ right moves, where $i$ varies between $0$ and $l$. A left move does not alter the number of digits and a right move increases the number of digits by $1$. As we start off with one digit at level 0, we find that the total number of digits in level $l$ is given by
\[ \sum_{i=0}^l (i+1){l \choose i} = (l+2)2^{-1}, \quad l\geq0. \]

Suppose that the $n$-th digit of $\C K$ occurs within the $L$-th level of the Kepler tree. The number of divided occurrences in the first $n$ digits of $\C K$ is then bounded from above by
\[ \sum_{l=0}^{L-1} k2^l + \C O(2^L) = k(2^L -1) + \C O(2^L). \]
Furthermore, the total number of possible occurrences of $d$ in the first $n$ digits of $\C K$ is
\[ \sum_{l=0}^{L-1} (l+2)2^{l-1} - k + 1 + \C O(2^L) = L2^{L-1} - k + 1 + \C O(2^L). \]
When we consider the asymptotic frequency of occurrences, we note that $n\to \infty$ implies that $L \to \infty$. Therefore the asymptotic frequency of this type of occurrences is
\[ 
\lim_{L \to \infty} \frac{ k(2^L -1) + \C O(2^L) } { L2^{L-1} - k + 1 + \C O(2^L)} = 0. 
\]
\end{proof}

\begin{theorem}
The number $\C K$, defined in \eqref{x}, is Minkowski normal.
\label{thm:main}
\end{theorem}
\begin{proof}
Let $d = d_1,d_2,\cdots, d_k$ be an arbitrary block of length $k$, $d_i \in \N$. In order to determine the frequency of $d$ in $\C K$ it is sufficient to count the binary blocks $10^{d_k-1} \cdots 1 0^{d_2-1} 1 0^{d_1 -1}1$ and $10^{d_k-1} \cdots 1 0^{d_2-1} 1 0^{d_1 -1} 0 $ in $\C C_2$. We argue this by considering the four different types of occurrences. \\

It follows from Lemma \ref{lemma:concatfreq} that the frequency of divided occurrences of $d$ tends to 0. \\

Now, let $p/q$ be an arbitrary rational in the Kepler tree that corresponds to the continued fraction $[a_1, a_2, \cdots, a_n]$. By 
Lemma \ref{lemma:rationaloccurrence}, the path from $1/2$ to $p/q$ is unique and given by
\[ L^{a_n-2} \cdots R L^{a_{2}-1}R L^{a_1 -1}. \]
Similarly, there exists a unique path to the rational $[d_1, d_2, \cdots, d_k, a_1, a_2, \cdots, a_n]$. By \eqref{eqn:keplerpath}, this path is
\[ \bm{L^{a_n-2} \cdots R L^{a_{2}-1} R L^{a_1 -1}}R L^{d_k-1} \cdots RL^{d_{2}-1} R L^{d_1 -1}. \]
Considering the latter path, we see that it passes through the rational $p/q$, of which the path is marked in bold. As this path and that to $p/q$ are unique, we conclude that there exists a unique subpath from $p/q$ to $[d_1, d_2, \cdots, d_k, a_1, a_2, \cdots, a_n]$ that is given by
\[ R L^{d_k-1} \cdots RL^{d_{2}-1} R L^{d_1 -1}. \] 
Therefore, the following binary code corresponds to $d$ occurring at the start of a continued fraction expansion
\[10^{d_k-1} \cdots 10^{d_2-1} 1 0^{d_1 -1}. \tag{A} \]

The binary code associated to occurrences of $d$ in the middle of a continued fraction expansion is similar. The difference with (A) is that another right move is needed in the Kepler tree. This preserves the block forever and causes it to occur in the middle. Therefore, the binary code associated to this type of occurrence is the same as that in (A) with a $1$ appended. Hence
\[ 10^{d_k-1} \cdots 10^{d_2-1} 1 0^{d_1 -1} 1. \tag{B} \]

Lastly we consider what happens when $d$ occurs at the end of a continued fraction. Due to the fact that the Kepler rule alters the start of continued fraction expansions, these type of occurrences are descendants from the rational $[d_1, d_2, \cdots, d_k]$. In order to preserve the block $d$, another right move is needed. Using this and Lemma \ref{lemma:rationaloccurrence} we find that the corresponding binary code is
\[ 0^{d_k-2} \cdots 1 0^{d_2-1} 1 0^{d_1 -1}1, \tag{C} \]
where the last 1 results from the extra right move. However, occurrences of this binary code in $\C C_2$ do not always correspond to an occurrence of $d$ in $\C K$. This is due to the fact that the digit $2$ is used to form $d_k$. That is, $d_k$ is formed from the digit $2$, whereas in the other type of occurrences, the block $d$ is formed from scratch. Hence for the binary code in (C) to correspond to an occurrence of $d$ in $\C K$, this occurrence of $d$ should originate from a rational of the form $[2, b_2, \cdots, b_{j-1}, b_j]$. By Lemma \ref{lemma:rationaloccurrence}, this corresponds to rationals that have a binary code given by
\[ 0^{b_j-2} \cdots 1 0^{b_2-1} 1 0. \]
In other words, for (C) to correspond to an occurrence of $d$ in $\C K$, we need to consider occurrences of $d$ that originate from rationals whose corresponding binary code ends in $10$. If $d$ is formed through a subpath that starts from such a rational, the binary code that is associated to this subpath is appended to that of the rational it originates from. We conclude that we can count these occurrences by looki ng at the frequency of the block
\[ 1 0 0^{d_k-2} \cdots 1 0^{d_2-1} 1 0^{d_1 -1}1 = 10^{d_k-1} \cdots 1 0^{d_2-1} 1 0^{d_1 -1}1. \tag{C*} \]
This is similar to (B). Moreover by counting the blocks in (A), we count (B) and (C*) as well. In order to prevent double counts, we append a $0$ to the code in (A). In conclusion, in order to find the frequency of $d$ in $\C K$, it is sufficient to consider the asymptotic frequencies of $10^{d_k-1} \cdots 1 0^{d_2-1} 1 0^{d_1 -1}1$ and $10^{d_k-1} \cdots 1 0^{d_2-1} 1 0^{d_1 -1} 0 $ in $\C C_2$. Both blocks occur with relative frequency
\[2^{-(d_1 + \cdots + d_k+1) }.\]
This results from the fact that the binary codes are of length $d_1+d_2+\cdots+d_k+1$ and that $\C C_2$ is normal in base $2$. Adding these frequencies gives the desired result
\[ \frac{1}{2^{d_1 + \cdots + d_k +1}} + \frac{1}{2^{d_1 + \cdots + d_k +1}} = 2^{-(d_1 + \cdots + d_k)}.\]
We conclude that $\C K$ is Minkowski normal.
\end{proof}

\section{Extending Minkowski normality}
When constructing a normal number, it is the ordering that is chosen that determines the distribution. Apparently, ordering the rationals based on their denominator leads to the distribution given by the Gauss measure, e.g. see Vandehey \cite{Vandehey2016a}. Although the sequence of rationals in \eqref{eqn:seqaks} is distributed according to the Lebesgue measure and not the Gauss, it is not that surprising that the number constructed by Adler, Keane and Smorodinsky is continued fraction normal. When we consider the frequency of occurrences of an arbitrary block $d=d_1,d_2, \cdots, d_k$ starting at the $n$-th position of a continued fraction expansion of a number in a uniformly distributed sequence, this frequency is given by the Lebesgue measure of the set $\C T^{-n}\Delta(d)$ \cite{Adler1981}, where $\C T$ denotes the Gauss map. Gauss showed that, as $n \to \infty$,  $\lambda(\C T^{-n}\Delta(d))$ converges in distribution to $\gamma(\Delta(d))$. In a similar manner, we can argue that $\C K$ should be Minkowski normal. Namely, the sequence of rationals that is obtained by ordering the rationals in the Kepler tree top-down left-right, see \eqref{eqn:seqrationals}, is distributed according to the Minkowski question mark. Then it follows that the frequency of occurrences of $d$, starting at the $n$-th position of a continued fraction expansion of a number in a Minkowski question mark distributed sequence, is given by the Minkowski measure of $\C T^{-n} \Delta(d)$. As $\mu_?$ is $\C T$-invariant, this measure is simply $\mu_?(\Delta(d))$. The fact that the sequence in \eqref{eqn:seqrationals} is distributed according to $\mu_?$ has implicitly been proved by Viader, Parad\'is and Bibiloni \cite{Viader1998}. In the article, they first define a one-to-one correspondence $q: \N \to (0,1)$. The first few terms of $q$ are
\begin{align*}
q(1) &= [2] = 1/2 & & q(5) = [1,3] = 3/4 \\
q(2) &= [3] = 1/3 & & q(6) = [2,2] = 2/5 \\
q(3) &= [1,2] = 2/3 & & q(7) = [1,1,2] = 3/5 \\
q(4) &= [4] = 1/4 & & q(8) = [5] = 1/5,
\end{align*}
which result from the following definition. If $n = 2^{a_1} + 2^{a_2} +\cdots + 2^{a_k}$ with \linebreak $0 \leq a_1 < a_2 < \cdots < a_k$, then:
\begin{equation}
q(n) := \begin{cases} [k+2] &\mbox{if } n = 2^k,\\
[a_1 +1, a_2 - a_1, a_3 - a_2, \cdots, a_k - a_{k-1} +1] &\mbox{otherwise.}
\end{cases}
\label{eqn:q(n)}
\end{equation}
Among other things, Viader, Parad\'is and Bibiloni prove that, for any $x \in [0,1]$,
\begin{equation}
\limn \frac{ \# \{ q(i) \leq x : 1 \leq i \leq n \}} {n} = ?(x),
\label{thm:Viader}
\end{equation}
see \cite[Theorem 2.7]{Viader1998}. Here $\# A$ denotes the cardinality of the set $A$.
We next show that the sequence of rationals in \eqref{eqn:seqrationals} is distributed according to the Minkowski question mark. More specifically, we prove that this sequence coincides with the sequence $(q(i))_{i\geq 1}$. Let the sequence in \eqref{eqn:seqrationals} be represented by $( k_i )_{i \geq 1}$. That is, $k_i$ denotes the $i$-th rational in \eqref{eqn:seqrationals}.
\begin{lemma}
The sequence $(k_i)_{i\geq_1}$ is distributed according to the Minkowski question mark measure. That is, for any $x \in [0,1]$, we have that
\[ \limn \frac{ \# \{ k_i \leq x : 1 \leq i \leq n \}} {n} = ?(x), \]
where $\# A$ denotes the cardinality of the set $A$.
\label{cor:k_i}
\end{lemma}
\begin{proof} We prove that $q(n) = k_n$ for all $n \in N$. It is clear that $q(1) = k_1 = 1/2$. We next show that the Kepler rule coincides with:
\begin{center}
\begin{tikzpicture}[scale = 0.55]
  \node {$q(n)$}[sibling distance=70mm]
    child {node {$q(2n)$}}
    child {node {$q(2n+1)$}};
\end{tikzpicture}
,
\end{center}
which concludes the proof. Let $n = 2^{a_1} + 2^{a_2} +\cdots + 2^{a_k}$ with $0 \leq a_1 < a_2 < \cdots < a_k$. Suppose that $n = 2^l$ for some $l$. Then $2n = 2^{l+1}$ and $2n +1 = 2^0 + ^{l+1}$. Using \eqref{eqn:q(n)}, we find
\begin{center}
\begin{tikzpicture}[scale = 1]
  \node {$q(n) = [l+2] $}[sibling distance=65mm]
    child {node {$q(2n) = [(l+1) + 2] = [(l+2) + 1]$}}
    child {node {$q(2n+1) = [0 + 1, (l+1) - 0 + 1] = [1, l+2].$}};
\end{tikzpicture}
\end{center}
Next, assume that $n = 2^{a_1} + 2^{a_2} + \cdots + 2^{a_k} \neq 2^l$. Then  $q(n) = [a_1 + 1, a_2 - a_1, a_3 - a_2, \cdots, a_k - a_{k-1} + 1]$, and
\begin{align*}
2n &= 2^{a_1 +1} + 2^{a_2 +1} + \cdots + 2^{a_k +1}; \\
2n +1 &= 2^0 + 2^{a_1 +1} + 2^{a_2 +1} + \cdots + 2^{a_k +1}.
\end{align*}
Applying \eqref{eqn:q(n)} to the above, we get
\begin{align*}
q(2n) &= [(a_1 + 1) + 1, (a_2+1)- (a_1 +1), (a_3+1) - (a_2+1), \cdots, (a_k+1) \\
 & \quad{} - (a_{k-1} +1)+ 1] \\
& = [(a_1 + 1) + 1, a_2 - a_1, a_3 - a_2, \cdots, a_k - a_{k-1} + 1] ; \\
q(2n +1) &= [0 + 1, (a_1 + 1) - 0, (a_2+1) - (a_1+1), (a_3+1)  \\ &\quad{} - (a_2+1), \cdots, (a_k+1) - (a_{k-1}+1) + 1] \\
& = [1, (a_1 + 1), a_2 - a_1, a_3 - a_2, \cdots, a_k - a_{k-1} + 1].
\end{align*}
We conclude that $(q(i))_{i\geq 1}$ coincides with $(k_i)_{i\geq 1}$. Therefore, there is an equivalence between the statement in \eqref{thm:Viader} and the limit in Lemma \ref{cor:k_i}.
\end{proof}

Thus, the sequence in \eqref{eqn:seqrationals} is distributed according to $\mu_?$. Apart from this fact, there is an important underlying structure in the sequence that causes normality. We discuss this structure and show that it can be used to construct a class of Minkowski normal numbers. Moreover, we provide an explicit example using the Farey tree. \\

The continued fraction normality of $x_{aks}$ results from the ordering of rationals based on their denominator. This ordering causes the sequence of rationals in \eqref{eqn:seqaks} to be distributed uniformly and hence $x_{aks}$ to be continued fraction normal. Minkowski normality of $\C K$, however, results from a completely different underlying structure. The underlying structure in this case comes from fact that the rationals are ordered increasingly, based on the sum of the digits of their continued fraction expansion. That is, the $l$-th level of the Kepler tree contains all possible rationals that have a continued fraction expansion whose sum of digits is equal to $l+2$. By ordering these top-down, left-right, the ordering is done as claimed. To see that the Kepler tree has this structure, we start by considering the root. The root of the tree, which corresponds to level $0$, is given by $1/2 = [2]$. Then, every next level, the sum of digits of the continued fraction expansion is increased by $1$ through the Kepler rule. Furthermore, the $l$-th level of the Kepler tree contains $2^l$ rationals, which is exactly the number of distinct\footnote{ We say that two rationals p/q and r/s are distinct if and only if $ps\neq qr$.} rationals that have a continued fraction expansion whose digits sum up to $l+2$.
\begin{lemma}
There exist exactly $2^l$ distinct rationals that have a continued fraction expansion of which the sum of the digits equals $l+2$, $l\geq 0$. That is,
\[ \# \Big\{ \frac{p}{q} \in [0,1): \frac{p}{q} = [a_1, a_2, \cdots, a_n], \sum_{i=1}^n a_i = l+2 \Big\} = 2^l, \]
where $\# A$ denotes the cardinality of the set $A$.
\end{lemma}
We omit a proof, as it follows directly from \cite[p. 215]{Viader1998}. Due to this lemma, we conclude that $\C K$ is a concrete example of a number that is obtained by concatenating the (reduced) continued fraction expansions of all rationals based on the sum of their digits, in increasing order. That is, one first concatenates the continued fraction expansions of rationals that have a continued fraction expansion of which the digits sum up to 2, then those that sum up to 3, etc. It turns out that all such constructions are Minkowski normal. In order to prove this, we use the fact that generalised Champernowne numbers are normal. That is, if we take $\C C_2$ and rearrange the blocks of the same length in any order, the resulting number is normal in base 2 \cite{Denker1992}. Due to the structure that underlies our construction, we can use this to extend our results. Again, the key idea is the unique correspondence between binary codes of length $l$ and continued fractions whose digits sum up to $l+2$. Let $[a_1, a_2, \cdots, a_n]$ be such that $\sum_{i=1}^n a_i = l+2$, then recall that this correspondence is given by
\begin{equation}
[a_1, a_2, \cdots, a_n] \overset{\text{binary code}}{\longleftrightarrow}  \underbrace{0^{a_n-2} \cdots 1 0^{a_2-1} 1 0^{a_1 -1}}_{\text{binary code of length l}}.
\label{eqn:cf-bin}
\end{equation}

The proof of Theorem \ref{thm:main} shows that we can count arbitrary blocks in $\C K$ through binary codes and explains why and how by referring to the structure of the Kepler tree. However, it is the coding that is important. Moreover, it is the explicit one-to-one correspondence between continued fraction expansions and binary codes that allows us to obtain frequencies and extend our results. This is due to the fact that divided occurrences are negligible and that the binary codes used in the proof result from the coding that is used. That is, if we convert a continued fraction expansion $[a_1,a_2, \cdots, a_n]$ to its binary code $0^{a_n-2} \cdots 1 0^{a_2-1} 1 0^{a_1 -1}$, we can use the binary codes in the proof to obtain the frequency of occurrences of $d$ in  $[a_1, a_2, \cdots, a_n]$. As such, we can extend the normality of $\C K$ to more general cases.

\begin{theorem}
Let the constructed number $\C K$ be denoted by
\begin{align*}
\C K = [ \kappa_1^1, \text{ } \kappa_2^1, \text{ } \kappa_1^2, \text{ } \kappa_2^2, \text{ } \kappa_3^2, \text{ } \kappa_4^2, \text{ } \kappa_1^3, \text{ } \cdots],
\end{align*}
where $\kappa_1^l, \kappa_2^l, \cdots, \kappa_{2^l}^l $ is the concatenation of the continued fraction expansions of the rationals in the $l$-th level of the Kepler tree, ordered from left to right. Furthermore, for all $l \in \N$, let $\pi_l$ be a permutation of $\{ 1, 2, \ldots, 2^l \}$. Then
\[ \C K^{\pi} := [
 \kappa_{\pi_1(1)}^1, \text{ } \kappa_{\pi_1(2)}^1,\text{ } \kappa_{\pi_2(1)}^2, \text{ } \kappa_{\pi_2(2)}^2, \text{ } \kappa_{\pi_2(3)}^2, \text{ } \kappa_{\pi_2(4)}^2, \text{ } \kappa_{\pi_3(1)}^3, \text{ } \cdots] \]
is Minkowski normal.
\label{thm:extendK}
\end{theorem}
\begin{proof}
Let $\C C_2$ be denoted by
\begin{align*}
\C C_2 = 0.
\text{ } c_1^1 \text{ } c_2^1\text{ } c_1^2 \text{ } c_2^2\text{ } c_3^2 \text{ } c_4^2 \text{ } c_1^3 \text{ } \cdots,
\end{align*}
where $c_1^l c_2^l \cdots c_{2^l}^l $ denotes the concatenation of all binary codes in the $l$-th level of the binary Kepler tree, ordered from left to right. It follows from \cite{Denker1992} and \cite{Shiokawa1975} that
\[ \C C_2^{\pi} := 0.
\text{ } c_{\pi_1(1)}^1 \text{ } c_{\pi_1(2)}^1\text{ } c_{\pi_2(1)}^2 \text{ } c_{\pi_2(2)}^2\text{ } c_{\pi_2(3)}^2 \text{ } c_{\pi_2(4)}^2 \text{ } c_{\pi_3(1)}^3 \text{ } \cdots \]
is normal in base $2$.
Let $d= d_1, d_2, \cdots, d_k$ be an arbitrary block of length $k$. Note that $\C C_2^{\pi}$ corresponds to the concatenation of the binary codes of the continued fraction expansions that are concatenated in $\C K^{\pi}$. As these binary codes and continued fraction expansions are (uniquely) related by the correspondence in \eqref{eqn:cf-bin}, we can count the number of occurrences of $d$ in $\C K^{\pi}$ by considering the frequency of $10^{d_k-1} \cdots 1 0^{d_2-1} 1 0^{d_1 -1}1$ and $10^{d_k-1} \cdots 1 0^{d_2-1} 1 0^{d_1 -1} 0 $ in $\C C_2^{\pi}$. The rest of the proof is analogous to the proof of Theorem \ref{thm:main}. We conclude that $\C K^{\pi}$ is Minkowski normal. 
\end{proof}

In particular, Theorem \ref{thm:extendK} proves Minkowski normality of the number that is obtained by concatenating the continued fraction expansions of the rationals in the Farey tree top-down, left-right. The tree starts with $1/2 = [2]$ at the root and forms new rationals according to the tree rule displayed in Figure \ref{fig:Fareytreerule}, see \cite{Isola2008}. 
\begin{figure}[!htb]
\begin{center}
\subfigure[ ]
{
\begin{tikzpicture}[scale = 0.55]
  \node {$[a_1,a_2,\cdots, a_n]$}[sibling distance=70mm]
    child {node {$[a_1,a_2,\cdots, (a_n+1)]$}}
    child {node {$[a_1,a_2,\cdots, (a_n-1), 2]$}};
\end{tikzpicture}
} 
\subfigure[ ]
{
\begin{tikzpicture}[scale = 0.55]
  \node {$[a_1,a_2,\cdots, a_n]$}[sibling distance=70mm]
    child {node {$[a_1,a_2,\cdots, (a_n-1), 2]$}}
    child {node {$[a_1,a_2,\cdots, (a_n+1)]$}};
\end{tikzpicture}
 }
 \caption{The rule of the Farey tree for (a) $n$ is odd and (b) $n$ is even.}
 \label{fig:Fareytreerule}
 \end{center}
\end{figure}
The ordering of the rationals that is obtained by this, is
\[ \frac{1}{2},\quad \frac{1}{3}, \frac{2}{3},\quad \frac{1}{4}, \frac{2}{5},\frac{3}{5}, \frac{3}{4},\quad \frac{1}{5}, \cdots. \]
It was implicitly shown by Kesseb\"omer and Stratmann \cite{Kessebohmer2008} that this sequence is distributed according to $\mu_?$. Therefore it should not be surprising that the following holds.

\begin{theorem}
The number that is obtained by concatenating the continued fraction expansions of the rationals in the Farey tree top-down left-right is Minkowski normal.
\label{cor:Fareynumber}
\end{theorem}

\begin{proof}
It can be seen from the tree rules that, regardless of whether $n$ is even or odd, the Farey tree rule increases the sum of the digits of the continued fraction expansion by 1 each next level. Therefore, the underlying structure of the tree is similar to that of the Kepler tree. Namely, the rationals are ordered increasingly, based on the sum of the digits of their continued fraction expansion. Hence, the $l$-th level of the Farey tree contains all possible rationals that have a continued fraction expansion whose sum of digits is equal to $l+2$. By concatenating the continued fraction expansions of the rationals in the Farey tree top-down, left-right, we obtain a permutation of $\C K$ that satisfies the conditions in Theorem \ref{thm:extendK}. Therefore, we conclude that the number that is obtained by concatenating the continued fraction expansions of the rationals in the Farey tree top-down left-right is Minkowski normal.
\end{proof}

\emph{Final remark}:
The extension in Theorem \ref{thm:extendK} is based on work of Shiokawa and Uchiyama \cite{Shiokawa1975}, which extends normality of the dyadic Champernowne number. Moreover, our extension is based on a specific case of \cite[Lemma 4]{Shiokawa1975}. This extension exploits the underlying structure of the Kepler tree to extend the Minkowski normality of $\C K$ to more general cases. As such, we preserve the underlying structure and hence - in some way - preserve normality. We have not been able to prove a full analogue of Shiokawa and Uchiyama's result. One of the reasons that we cannot extend normality to this general case, is that the we can no longer use the normality of $\C C_2^{\pi}$ to count frequencies. That is, our extension allows one to reorder the continued fraction expansions of rationals that have a continued fraction expansion of which the partial quotients sum up to the same number. A full analogue of the work of Shiokawa and Uchiyama would allow one to break up the continued fraction expanion of the same rationals into smaller parts and reorder these arbitrarily. However, when we break up continued fraction expansions into smaller parts, one creates subblocks of which the sum of its digits will vary and the composition of binary codes will change. Consider for instance the continued fraction $[2,1,1,3]$, which corresponds to the binary code $01110$. Suppose we break this up into $[2]$ and $[1,1,3]$. Then these correspond to the binary codes $\emptyset$ and $011$ respectively. Conversely, break up $01110$ into the blocks $011$ and $10$. These binary codes correspond, respectively, to the continued fraction expansions $[1,1,3]$ and $[1,3]$. This shows that the underlying structure is not preserved when breaking up continued fraction expansions into smaller parts. However, it should be possible to find a similar extension.

\bibliographystyle{amsplain}
\bibliography{library}

\end{document}